\documentclass[12pt]{article}

\usepackage{amsmath, amssymb,amsthm}

\usepackage{color}
\definecolor{darkred}{RGB}{139,0,0}
\definecolor{darkgreen}{RGB}{0,100,0}
\definecolor{darkmagenta}{RGB}{139,0,139}
\definecolor{darkpurple}{RGB}{110,0,180}
\definecolor{darkblue}{RGB}{40,0,200}
\definecolor{darkorange}{RGB}{255,140,0}

\newcommand{\bsx}{\boldsymbol{x}}

\newcommand{\bszero}{\boldsymbol{0}}
\newcommand{\bsj}{\boldsymbol{j}}
\newcommand{\bsm}{\boldsymbol{m}}

\newcommand{\bsy}{\boldsymbol{y}}
\newcommand{\bst}{\boldsymbol{t}}
\newcommand{\rd}{\,{\rm d}}

\newcommand{\bmo}{{\rm BMO}}

\newcommand{\NN}{\mathbb{N}}

\newcommand{\DD}{\mathbb{D}}

\newcommand{\cP}{\mathcal{P}}

\newtheorem{theorem}{Theorem}
\newtheorem{remark}{Remark}
\newtheorem{lemma}{Lemma}

\title{The BMO-discrepancy suffers from the curse of dimensionality}

\author{Friedrich Pillichshammer\thanks{The author is supported by the Austrian Science Fund (FWF): Project F5509-N26, which is a part of the Special Research Program ``Quasi-Monte Carlo Methods: Theory and Applications''.}}

\date{}

\begin{document}

\maketitle

\begin{abstract}
We show that the minimal discrepancy of a point set in the $d$-dimensional unit cube with respect to the BMO seminorm suffers from the curse of dimensionality.
\end{abstract}

\centerline{\begin{minipage}[hc]{130mm}{
}
\end{minipage}}

\vspace*{1cm}

The discrepancy is a quantitative measure for the irregularity of distribution of an $N$-element point set $\cP = \{\bsx_1, \bsx_2, \ldots, \bsx_N\}$ in the unit cube $[0,1)^d$ (see, e.g., \cite{BC,DT,kuinie,matou}). This concept has important applications in numerical analysis, where so-called Koksma-Hlawka inequalities establish a deep connection between norms of the discrepancy and worst-case errors of quasi-Monte Carlo rules (see, e.g, \cite{DKP22,DP10,nie92,NW10}). 

For a measurable subset $B \subseteq [0,1)^d$ the {\it local discrepancy} $\Delta_{\cP}$ defined as  
\begin{equation*}
\Delta_{\cP}(B) := \frac{|\{ j \in \{1, 2, \ldots, N\} \ : \ \bsx_j \in B \}| }{N} - \lambda_d(B),
\end{equation*}
where $\lambda_d$ stands for the $d$-dimensional Lebesgue measure. Choosing a suitable class of ``test sets'' $B$ and taking a norm $\|\cdot \|_{\bullet}$ of the local discrepancy with respect to the considered test sets leads to a discrepancy $\|\Delta_{\cP} \|_{\bullet}$ of the point set $\cP$. An important choice is the class of subintervals of $[0,1)^d$ that are anchored in the origin, i.e., intervals of the form $[\bszero,\bst)=[0,t_1)\times \ldots \times [0,t_d)$, where $\bst=(t_1,\ldots,t_d) \in [0,1]^d$. Taking the $L_p$-norm of the function $\bst \mapsto \Delta_{\cP}([\bszero,\bst))$ gives the usual {\it (star) $L_p$-discrepancy} $$L_{p,N}(\cP):=\|\Delta_{\cP}([\bszero,\cdot))\|_{L_p}=\left(\int_{[0,1]^d} |\Delta_{\cP}([\bszero,\bst))|^p \rd \bst\right)^{1/p} \quad \mbox{for $1\le p < \infty$}$$ with the usual adaptions for $p=\infty$, i.e., $$L_{\infty,N}(\cP):= \|\Delta_{\cP}([\bszero,\cdot))\|_{L_\infty}=\sup_{\bst \in [0,1]^d}|\Delta_{\cP}([\bszero,\bst))|.$$ The star $L_{\infty}$-discrepancy is often simply called {\it star-discrepancy} and is denoted by $D_N^{\ast}$. 

Another example are arbitrary subintervals of $[0,1)^d$ as test sets of the form $[\bsx,\bsy)=[x_1,y_1)\times \ldots \times [x_d,y_d)$ with $\bsx \le \bsy$, where $\bsx=(x_1,\ldots,x_d)$ and $\bsy=(y_1,\ldots,y_d)$ in $[0,1]^d$ and where $\bsx \leq \bsy$ means $x_i\leq y_i$ for all $i \in \{1,\ldots,d\}$. Again, taking the $L_p$-norm with respect to $\bsx \le \bsy$ leads to the so-called {\it extreme (or unanchored) $L_p$-discrepancy} 
\begin{equation}\label{def:extrL2}
L_{p,N}^{\mathrm{extr}}(\cP):=\left(\int_{[0,1]^d}\int_{[0,1]^d,\, \bsx\leq \bsy} |\Delta_{\cP}([\bsx,\bsy))|^p\rd \bsx\rd\bsy\right)^{1/p},
\end{equation}
again with the usual adaptions if $p=\infty$. Choosing so-called ``periodic boxes'' as test sets leads to the notion of periodic $L_p$-discrepancy (or diaphony in the special case $p=2$). See \cite{HKP20,KrPi22} for more details.

Besides intervals as test sets and the $L_p$-norms also other norms of the discrepancy function are considered and studied (see, e.g., \cite{BC,matou,NW10,T10}). In this context the exponential Orlicz norms and the BMO (semi)norm attracted much attention in recent years (see \cite{B11,BLPV,BM15,DHMP,DHPP}). 

Often discrepancies are directly related to the worst-case integration error of quasi-Monte Carlo rules for a suitable class of integrands (see, e.g., \cite[Chapter~9]{NW10}). For this reason, point sets with low discrepancy have an important application in numerical analysis. This motivates the study of the {\it $N$-th minimal discrepancy} with respect to the norm $\|\cdot \|_{\bullet}$ in dimension $d$ which is the best possible discrepancy over all $N$-point sets in $[0,1]^d$, i.e.,
\begin{equation*}
\mathrm{disc}_{\bullet}(N,d) := \inf_{\stackrel{\cP \subseteq [0,1]^d}{|\cP| = N}} \| \Delta_{\cP}\|_{\bullet}.
\end{equation*}
In many cases, for fixed dimension $d$, asymptotically optimal bounds on $\mathrm{disc}_{\bullet}(N,d)$ in $N$ are available. However, in practical integration problems, the dimension $d$ may be very large and then these optimal bounds are often useless, since they yield no useful information within the pre-asymptotic regime of $N$. This problem is studied in the discipline ``information based complexity''. 

The so-called {\it initial discrepancy} is defined as the discrepancy of the empty point set $\| \Delta_{\emptyset} \|_{\bullet}$. Since this quantity may depend on the dimension, it is used to normalize the $N$-th minimal discrepancy when the dependence of $\mathrm{disc}_{\bullet}(N,d)$ on $d$ (and some error threshold $\varepsilon$) is studied. Therefore the {\it inverse of the $N$-th minimal discrepancy} in dimension $d$ is defined as the number $N_{\bullet}(\varepsilon, d)$ to be the smallest number $N$ such that a point set with $N$ points exists that reduces the initial discrepancy at least by a factor of $\varepsilon\in(0,1)$, i.e.,
\begin{equation*}
N_{\bullet}(\varepsilon, d) := \min \big\{ N \in \mathbb{N}\,:\, \mathrm{disc}_{\bullet}(N, d) \le \varepsilon \|\Delta_\emptyset \|_{\bullet} \big\}.
\end{equation*}

Now a discrepancy is said to {\it suffer from the curse of dimensionality} if it has the unfavourable property that its inverse grows exponentially with the dimension $d$, i.e., if there exist numbers $C, \tau \in (0,\infty)$ and $\varepsilon_0 \in (0,1)$ such that $$N_{\bullet}(\varepsilon, d) \ge C(1+\tau)^d \quad \mbox{for all $\varepsilon \in (0,\varepsilon_0)$ and infinitely many $d\in \NN$.}$$ Otherwise, if this is not the case, one says that the discrepancy is {\it tractable}. There are many notions of tractability, which characterize the growth rate of $N_{\bullet}(\varepsilon, d)$ when $\varepsilon \rightarrow 0$ and $d \rightarrow \infty$, which in any case has to be sub-exponential. An important notion is {\it polynomial tractability}, which holds whenever there are numbers $C, \tau, \sigma \in(0,\infty)$ such that
\begin{equation*}
N_{\bullet}(\varepsilon, d) \le C\, d^\tau \varepsilon^{-\sigma} \quad \mbox{for all } \varepsilon \in (0,1) \mbox{ and all } d \in \mathbb{N}.
\end{equation*}
Another, weaker notion is {\it weak tractability}, which means that  
\begin{equation*}
\lim_{d+\varepsilon^{-1}  \to \infty} \frac{\log N_{\bullet}(\varepsilon, d)}{d+\varepsilon^{-1}} = 0.
\end{equation*}

The subject of tractability of multivariate problems and in particular of discrepancy is a very popular and active area of research. We refer to the books \cite{NW08,NW10,NW12} by Novak and Wo\'{z}niakowski for an introduction and an exhaustive exposition.

It is known from a famous result by Heinrich, Novak, Wasilkowski, and Wo\'{z}niakowski~\cite{HNWW} that the star $L_\infty$-discrepancy is polynomially tractable. On the other hand, the star $L_2$-discrepancy is known to suffer from the curse of dimensionality, as shown by Wo\'{z}niakowski \cite{W99} (see also \cite{NW10}). Also the extreme- (see \cite[Section~10.5.3, p.~94]{NW10}) and the periodic- (see \cite{DHP}) $L_2$-discrepancy suffer from the curse of dimensionality. The behavior of the $L_p$-discrepancies in between, where $p \notin \{2, \infty\}$, seems to be unknown.

It is now a natural and instructive question to ask what happens in intermediate spaces ``close'' to $L_{\infty}$. Standard examples of such spaces are the space $\bmo$, which stands for {\it bounded mean oscillation}, or exponential Orlicz spaces. The study of these spaces in discrepancy theory with the aim to understand the precise nature of the kink that occurs at the passage from the $L_p$ with finite $p$ to the $L_\infty$ norm was initiated by Bilyk, Lacey, Parissis, and Vagharshakyan in \cite{BLPV}. Other papers in this context followed, we refer to \cite{B11,BLPV,BM15,DHMP,DHPP}.

The dependence of the inverse of discrepancy with respect to Orlicz norms on the dimension has been studied recently in \cite{DHPP} with quite positive results showing that the discrepancy with respect to suitable Orlicz norms can exhibit both polynomial and weak tractability (see \cite[Theorem~1 and 2]{DHPP}). In this note we supplement the BMO case.

For the definition we need the concept of Haar functions: Let $\NN_0=\NN \cup \{0\}$ and $\NN_{-1}=\NN_0\cup\{-1\}$. Let $\DD_j = \{0,1,\ldots, 2^j-1\}$ for $j \in \NN_0$ and $\DD_{-1} = \{0\}$. For $\bsj = (j_1,\dots,j_d)\in\NN_{-1}^d$ let $\DD_{\bsj} = \DD_{j_1}\times\ldots\times \DD_{j_d}$. For $\bsj\in\NN_{-1}^d$ we write $|\bsj| = \max(j_1,0) + \cdots + \max(j_d,0)$.

For $j \in \NN_0$ and $m \in \DD_j$ we call the interval $I_{j,m} = \big[ 2^{-j} m, 2^{-j} (m+1) \big)$ the $m$-th dyadic interval in $[0,1)$ on level $j$. We put $I_{-1,0}=[0,1)$ and call it the $0$-th dyadic interval in $[0,1)$ on level $-1$. Let $I_{j,m}^+ = I_{j + 1,2m}$ and $I_{j,m}^- = I_{j + 1,2m+1}$ be the left and right half of $I_{j,m}$, respectively. For $\bsj \in \NN_{-1}^d$ and $\bsm = (m_1, \ldots, m_d) \in \DD_{\bsj}$ we call $I_{\bsj,\bsm} = I_{j_1,m_1} \times \ldots \times I_{j_d,m_d}$ the $\bsm$-th dyadic interval in $[0,1)^d$ on level $\bsj$. We call the number $|\bsj|$ the order of the dyadic interval $I_{\bsj,\bsm}$. Its volume is $2^{-|\bsj|}$.

Let $j \in \NN_{0}$ and $m \in \DD_j$. Let $h_{j,m}$ be the function on $[0,1)$ with support in $I_{j,m}$ and the constant values $1$ on $I_{j,m}^+$ and $-1$ on $I_{j,m}^-$. We put $h_{-1,0} = \mathbf{1}_{I_{-1,0}}$ on $[0,1)$ (where $\mathbf{1}_A$ denotes the indicator function of a set $A$). The function $h_{j,m}$ is called the {\it $m$-th dyadic Haar function on level $j$}.

Let $\bsj \in \NN_{-1}^d$ and $\bsm \in \DD_{\bsj}$. The function $h_{\bsj,\bsm}$ given as the tensor product
\[ h_{\bsj,\bsm}(\bsx) = h_{j_1,m_1}(x_1) \cdots h_{j_d,m_d}(x_d) \]
for $\bsx = (x_1, \ldots, x_d) \in [0,1)^d$ is called a {\it dyadic Haar function} on $[0,1)^d$. The system of dyadic Haar functions $h_{\bsj,\bsm}$ for $\bsj \in \NN_{-1}^d, \, \bsm \in \DD_{\bsj}$ is called {\it dyadic Haar basis on $[0,1)^d$}.

It is well known that the system
\[ \left\{2^{\frac{|\bsj|}{2}}h_{\bsj,\bsm} \,:\,\bsj\in\NN_{-1}^d,\,\bsm\in \DD_{\bsj}\right\} \]
is an orthonormal basis of $L_2([0,1)^d)$. For any function $f\in L_2([0,1)^d)$ we have Parseval's identity
\[ \|f\|_{L_2}^2 = \sum_{\bsj \in \NN_{-1}^d} 2^{|\bsj|} \sum_{\bsm\in \DD_{\bsj}}|\langle f,h_{\bsj,\bsm}\rangle|^2, \] where $\langle \cdot , \cdot \rangle$ denotes the usual $L_2$-inner product, i.e., $\langle f ,g \rangle =\int_{[0,1]^d} f(\bsx) g(\bsx) \rd \bsx$. The terms $\langle f,h_{\bsj,\bsm}\rangle$ are called the {\it Haar coefficients} of the function $f$.

For an integrable function $f:[0,1]^d \rightarrow \mathbb{R}$ define 
\begin{equation}\label{eq:defbmo}
\|f\|^2_{\bmo} :=\sup_{U \subseteq [0,1)^d} \frac{1}{\lambda_d(U)} \sum_{\bsj \in \mathbb{N}_0^d} 2^{|\bsj|} \sum_{\bsm \in \DD_{\bsj}\atop I_{\bsj,\bsm} \subseteq U}|\langle f, h_{\bsj,\bsm}\rangle|^2,
\end{equation} 
where the supremum is taken over all measurable sets $U \subseteq [0,1)^d$. 

In order to give an intuition, consider $d=1$ and assume that $U \subseteq [0,1)$ is a dyadic interval. Let $\langle f \rangle_U:=\lambda_1(U)^{-1} \int_U f(x)\rd x$ be the mean of $f$ over $U$. Then it is easy to show that $\langle (f- \langle f\rangle_U) {\bf 1}_U, h_{-1,0}\rangle=0$ and, for $j \in \NN_0$ and $m \in \DD_j$, $$\langle (f- \langle f\rangle_U) {\bf 1}_U , h_{j,m}\rangle = \left\{ 
\begin{array}{ll}
\langle f, h_{j,m}\rangle & \mbox{if $I_{j,m} \subseteq U$,}\\
0 & \mbox{otherwise,}
\end{array}
\right.$$
where ${\bf 1}_U$ is the indicator function of $U$. Hence we have by Parseval's identity, that 
\begin{eqnarray*}
\frac{1}{\lambda_1(U)}\int_U|f(x)- \langle f \rangle_U|^2 \rd x & = &  \frac{1}{\lambda_1(U)}\int_0^1|(f(x)- \langle f \rangle_U){\bf 1}_U(x)|^2 \rd x\\
& = & \frac{1}{\lambda_1(U)} \sum_{j=0}^{\infty} 2^j \sum_{m \in \DD_j \atop I_{j,m}\subseteq U} |\langle f, h_{j,m}\rangle|^2,
\end{eqnarray*}
where the right hand side is exactly the term that appears in right hand side of \eqref{eq:defbmo}. 

The BMO space contains all integrable functions $f$ with finite norm $\|f\|_{\bmo}$. Note that strictly speaking $\|f\|_{\bmo}$ is only a seminorm, since it vanishes on linear combinations of functions which are constant in one or more coordinate directions. This means that formally we need to consider a factor space over such functions. We also note that there are different definitions of BMO spaces in the multivariate case. The setting introduced here is the so-called dyadic product BMO as introduced by Bernard~\cite{be} and as usually studied in discrepancy theory (see, e.g., \cite{B11,BLPV,BM15,DHMP}). As mentioned in \cite{B11}, with the present definition of the BMO norm the famous $H^1$ - BMO duality is preserved, where $H^1$ stands for the Hardy space\footnote{I.e., the space of functions $f \in L_1$ with integrable Littlewood-Paley square function.} and just as $H^1$ often serves as a natural substitute of the $L_1$-space, in many problems of harmonic analysis the BMO-space naturally replaces the $L_{\infty}$-space. For a more detailed study of BMO spaces see \cite{CF80} and \cite{CWW85}.

We consider the $\bmo$-seminorm of the discrepancy function $\Delta_{\cP}$ with respect to anchored subintervals of $[0,1)^d$ and call this the {\it $\bmo$-discrepancy}. The inverse of this $\bmo$-discrepancy is denoted by $N_{\bmo}(\varepsilon,d)$ for $d \in \NN$ and $\varepsilon \in (0,1)$. 

In \cite{BLPV,BM15,DHMP} Roth-type lower bounds on the BMO-discrepancy are shown (see Remark~\Ref{re:rt} at the end of the present note) which show that in the context of discrepancy the BMO-norm behaves more like $L_p$ with finite $p$ rather than $L_{\infty}$. This phenomenon can be also observed in the context of dependence of the inverse of discrepancy on the dimension. It is the aim of this note to show that also there the BMO-norm behaves more like the $L_2$-case\footnote{The general $L_p$-case for finite $p\not=2$ is still open.} rather than $L_{\infty}$.

\begin{theorem}\label{thm1}
For $d \in \NN$ and $\varepsilon \in (0,1)$ we have $$N_{\bmo}(\varepsilon,d) \ge \left(\frac{4}{3}\right)^d (1-\varepsilon^2).$$ In particular, the BMO-discrepancy suffers from the curse of dimensionality.
\end{theorem}

So the inverse of BMO-discrepancy behaves like the inverse of the classical (star) $L_2$-discrepancy rather than the inverse of the star $L_\infty$-discrepancy (for which we have polynomial tractability according to \cite{HNWW}). The proof of Theorem~\ref{thm1} is not a big deal because we can fall back on powerful auxiliary results. In particular, we use a relation between the BMO-discrepancy and the extreme $L_2$-discrepancy which is interesting on its own and which follows from a recent Haar series expansion of the extreme $L_2$-discrepancy. But first we compute the initial BMO-discrepancy.

\begin{lemma}
The initial BMO-discrepancy in dimension $d$ is $\|\Delta_{\emptyset}\|_{\bmo}=12^{-d/2}$.
\end{lemma}

\begin{proof}
For $\bsj \in \NN_0^d$ and $\bsm \in \DD_{\bsj}$ we have
\begin{eqnarray*}
\langle \Delta_{\emptyset},h_{\bsj,\bsm} \rangle & = & \prod_{i=1}^d \int_0^1  x \, h_{j_i,m_i}(x) \rd x \\
& = &  \prod_{i=1}^d \left(\int_{m_i/2^{j_i}}^{(m_i+1/2)/2^{j_i}} x \rd x - \int_{(m_i+1/2)/2^{j_i}}^{(m_i+1)/2^{j_i}} x \rd x\right) = \frac{(-1)^d}{2^{2 (d+|\bsj|)}}.
\end{eqnarray*} 
Hence, considering $U=[0,1)^d$, we find that
\begin{eqnarray*}
\|\Delta_{\emptyset}\|_{\bmo}^2 \ge  \sum_{\bsj \in \mathbb{N}_0^d} 2^{|\bsj|} \sum_{\bsm \in \DD_{\bsj}} \frac{1}{2^{4 (d+|\bsj|)}} = \sum_{\bsj \in \mathbb{N}_0^d}   \frac{2^{2 |\bsj|}}{2^{4 (d+|\bsj|)}}=\left(\frac{1}{16} \sum_{j=0}^{\infty} \frac{1}{4^j}\right)^d=\frac{1}{12^d}.
\end{eqnarray*}
On the other hand, it is clear that for each $U \subseteq [0,1)^d$ and $\bsj \in \NN_0^d$ there are at most $2^{|\bsj|} \lambda_d(U)$ values of $\bsm \in \DD_{\bsj}$ such that $I_{\bsj,\bsm} \subseteq U$. Hence we have
\begin{eqnarray*}
\|\Delta_{\emptyset}\|_{\bmo}^2 \le \sup_{U \subseteq [0,1)^d} \frac{1}{\lambda_d(U)} \sum_{\bsj \in \mathbb{N}_0^d} 2^{|\bsj|}  2^{|\bsj|} \lambda_d(U)   \ \frac{1}{2^{4 (d+|\bsj|)}} = \sum_{\bsj \in \mathbb{N}_0^d}   \frac{2^{2 |\bsj|}}{2^{4 (d+|\bsj|)}} =\frac{1}{12^d}
\end{eqnarray*}
and the result follows.
\end{proof}

In \cite[Proposition~3]{KrPi22} a Haar-series expansion of the extreme $L_2$-discrepancy (see \eqref{def:extrL2} with $p=2$) is presented, which we state in the next lemma.  

\begin{lemma} \label{le2}
For every $N$-element point set $\cP$ in $[0,1)^d$ we have 
\begin{equation}\label{l2exthaar}
(L_{2,N}^{\mathrm{extr}}(\cP))^2=\sum_{\bsj \in \NN_0^d} 2^{|\bsj|}\sum_{\bsm\in\DD_{\bsj}}|\langle \Delta_{\cP}([\bszero,\cdot)),h_{\bsj,\bsm}\rangle|^2.
\end{equation}
\end{lemma}

Note that in formula \eqref{l2exthaar} the Haar coefficients of the local discrepancy of {\it anchored} intervals appear. Comparing formula \eqref{l2exthaar} with the corresponding formula for the star $L_2$-discrepancy, which is $$(L_{2,N}^{\mathrm{star}}(\cP))^2=\sum_{\bsj \in \NN_{-1}^d} 2^{|\bsj|}\sum_{\bsm\in\DD_{\bsj}}|\langle \Delta_{\cP}([\bszero,\cdot)),h_{\bsj,\bsm}\rangle|^2,$$ it can be observed that the only difference is that in the case of extreme $L_2$-discrepancy the terms of order $-1$ are not present. This also shows that the extreme $L_2$-discrepancy is dominated by the star $L_2$-discrepancy.

For us, however, the following is important. Comparing \eqref{l2exthaar} with the definition of the BMO seminorm \eqref{eq:defbmo} for the local discrepancy and taking $U=[0,1)^d$ immediately implies the following lemma, which is interesting on its own and which is the key to the proof of Theorem~\ref{thm1}.

\begin{lemma}\label{le3}
For every $N$-element point set $\cP$ in $[0,1)^d$ we have $$\|\Delta_{\cP}\|_{\bmo} \ge L_{2,N}^{\mathrm{extr}}(\cP).$$
\end{lemma}

Now we can give the proof of Theorem~\ref{thm1}.

\begin{proof}[Proof of Theorem~\ref{thm1}]
It is easily shown (or see \cite[p.~33]{NW10}) that also the initial extreme $L_2$-discrepancy equals $12^{-d/2}$. This shows that $\|\Delta_{\emptyset}\|_{\bmo}=L_{2,0}^{\mathrm{extr}}(\emptyset)$. Hence, using Lemma~\ref{le3},
\begin{eqnarray*}
N_{\bmo}(\varepsilon,d) & = & \min\{N \in \NN \ : \ {\rm disc}_{\bmo}(N,d) \le \varepsilon \|\Delta_{\emptyset}\|_{\bmo}\}\\
& \ge &  \min\{N \in \NN \ : \ {\rm disc}_{L_2^{{\rm extr}}}(N,d) \le \varepsilon L_{2,0}^{\mathrm{extr}}(\emptyset)\}=  N_{L_2^{{\rm extr}}}(\varepsilon,d).
\end{eqnarray*}
For the latter quantity however it follows from \cite[Section~10.5.3, p.~94]{NW10} that $$N_{L_2^{{\rm extr}}}(\varepsilon,d) \ge \left(\frac{9}{4}\right)^d(1-\varepsilon^2).$$ Hence the result follows.
\end{proof}

\begin{remark}\label{re:rt}\rm
Another interesting consequence of Lemma~\ref{le3} is the following. From \cite[Theorem~6]{HKP20} (or also  \cite[Corollary~4]{KrPi22}) it is known that for every $d \in \NN$ there exists a positive real $c_d$ with the property that for every $N$-element point set $\cP$ in $[0,1)^d$ we have $$L_{2,N}^{{\rm extr}}(\cP) \ge c_{d} \frac{(1+\log N)^{\frac{d-1}{2}}}{N}.$$ Now from Lemma~\ref{le3} it follows that the same lower bound applies to the BMO-discrepancy for every $N$-element point set in $[0,1)^d$. Thus, both discrepancies satisfy a Roth-type lower bound like the usual (star) $L_2$-discrepancy (see \cite{Roth}). The Roth-type lower bound for the BMO-discrepancy has been proved earlier by Bilyk, Lacey, Parissis and Vagharshakyan~\cite[Theorem~1.6]{BLPV} for $d=2$ and by Bilyk and Markhasin~\cite[Theorem~1.2]{BM15} for $d \ge 3$.
\end{remark}

\vspace{0.5cm}
\noindent{\bf Author's Address:}\\
\noindent Friedrich Pillichshammer, Institut f\"{u}r Finanzmathematik und Angewandte Zahlentheorie, Universit\"{a}t Linz, Altenbergerstra{\ss}e 69, A-4040 Linz, Austria. 

\noindent Email: friedrich.pillichshammer@jku.at


\begin{thebibliography}{99.}
\bibitem{BC} J.~Beck and W.W.L.~Chen: {\it Irregularities of Distribution.} Cambridge University Press, Cambridge, 1987.

\bibitem{be} A.~Bernard: Espaces {$H^{1}$} de martingales \`a deux indices. {D}ualit\'{e} avec les martingales de type ``{BMO}''.  Bull. Sci. Math. (2) 103(3): 297-303, 1979.

\bibitem{B11} D.~Bilyk: On Roth's orthogonal function method in discrepancy theory. Unif. Distrib. Theory 6(1): 143-184, 2011.
 
\bibitem{BLPV} D.~Bilyk, M.T.~Lacey, I.~Parissis, and A.~Vagharshakyan: Exponential squared integrability of the discrepancy function in two dimensions. Mathematika 55(1): 1-27, 2009.

\bibitem{BM15} D.~Bilyk and L.~Markhasin: BMO and exponential Orlicz space estimates of the discrepancy function in arbitrary dimension.  J. Anal. Math. 135(1): 249-269, 2018.

\bibitem{CF80} S.-Y.A.~Chang and R.~Fefferman: A continuous version of duality of $H^1$ with BMO on the bidisc. Ann. of Math. 112: 179-201, 1980.

\bibitem{CWW85} S.-Y.A.~Chang, J.M.~Wilson, and T.H.~Wolff: Some weighted norm inequalities concerning the Schr\"odinger operators. Comment. Math. Helv. 60: 217-246, 1985.

\bibitem{DHMP} J.~Dick, A.~Hinrichs, L.~Markhasin, and F.~Pillichshammer: Discrepancy of second order digital sequences in function spaces with dominating mixed smoothness. Mathematika 63(3): 863-894, 2017.

\bibitem{DHP}  J.~Dick, A.~Hinrichs, and F.~Pillichshammer: A note on the periodic $L_2$-discrepancy of Korobov's $p$-sets. Arch. Math. (Basel) 115(1): 67-78, 2020.

\bibitem{DHPP} J.~Dick, A.~Hinrichs, F.~Pillichshammer, and J. Prochno: Tractability properties of the discrepancy in Orlicz norms. J. Complexity 61, paper ref. 101468, 9 pp., 2020.

\bibitem{DKP22} J.~Dick, P.~Kritzer, and F.~Pillichshammer: {\it Lattice Rules--Numerical Integration, Approximation, and Discrepancy.} Springer Series in Computational Mathematics 58, Springer, Cham, 2022.

\bibitem{DP10} J.~Dick and F.~Pillichshammer: {\it Digital Nets and Sequences--Discrepancy Theory and Quasi-Monte Carlo Integration.} Cambridge University Press, Cambridge, 2010.

\bibitem{DT} M.~Drmota and R.F.~Tichy: {\it Sequences, Discrepancies and Applications.} Lecture Notes in Mathematics, vol. 1651, Springer, Berlin, 1997.

\bibitem{HNWW} S.~Heinrich, E.~Novak, G.W.~Wasilkowski, and H.~Wo\'{z}niakowski: The inverse of the star-discrepancy depends linearly on the dimension. Acta Arith. 96(3): 279-302, 2001.

\bibitem{HKP20} A.~Hinrichs, R.~Kritzinger, and F.~Pillichshammer: Extreme and periodic $L_2$ discrepancy of plane point sets.  Acta Arith. 199(2): 163-198, 2021. 

\bibitem{KrPi22} R.~Kritzinger and F.~Pillichshammer: Point sets with optimal order of extreme and periodic discrepancy. Acta Arith. 204(3): 191-223, 2022. 

\bibitem{kuinie} L.~Kuipers and H.~Niederreiter: {\it Uniform Distribution of Sequences.} John Wiley, New York, 1974.

\bibitem{matou} J.~Matou\v{s}ek: {\it Geometric Discrepancy}. Algorithms and Combinatorics 18, Springer Verlag, Berlin, 1999.

\bibitem{nie92} H.~Niederreiter: {\it Random Number Generation and Quasi-Monte Carlo Methods}. No. 63 in CBMS-NSF Series in Applied Mathematics, SIAM, Philadelphia, 1992.

\bibitem{NW08} E.~Novak and H.~Wo\'zniakowski: {\it Tractability of Multivariate Problems, Volume I: Linear information.} EMS Tracts in Mathematics 6, European Mathematical Society, Z\"urich, 2008.

\bibitem{NW10} E.~Novak and H.~Wo\'zniakowski: {\it Tractability of Multivariate Problems, Volume II: Standard Information for Functionals.}  EMS Tracts in Mathematics 12, European Mathematical Society, Z\"urich, 2010.

\bibitem{NW12} E.~Novak and H.~Wo\'zniakowski: {\it Tractability of Multivariate Problems, Volume III: Standard Information for Operators.} EMS Tracts in Mathematics 18, European Mathematical Society, Z\"urich, 2012.

\bibitem{Roth} K.F.~Roth: On irregularities of distribution. Mathematika 1: 73-79, 1954.

\bibitem{T10} H.~Triebel: {\it Bases in Function Spaces, Sampling, Discrepancy, Numerical Integration.} EMS Tracts in Mathematics 11, European Mathematical Society, Z\"urich, 2010.

\bibitem{W99} H.~Wo\'{z}niakowski: Efficiency of quasi-Monte Carlo algorithms for high dimensional integrals. In: {\it Monte Carlo and Quasi-Monte Carlo Methods 1998.} (H.~Niederreiter and J.~Spanier, eds.), pp. 114-136, Springer-Verlag, Berlin, 1999.

\end{thebibliography}
\end{document}